\documentclass[12pt]{amsart}
\usepackage{amssymb}

\newtheorem{theorem}{Theorem}[section]

\newcommand\relphantom[1]{\mathrel{\phantom{#1}}}

\title{(Weak) compactness of Hankel operators on $BMOA$}
\author{Michael Papadimitrakis}
\address{Department of Mathematics, University of Crete, Knossos Ave., 71409 Iraklio, Greece}
\email{papadim@math.uoc.gr}

\begin{document}
{\allowdisplaybreaks
\begin{abstract}
We prove that the notions of compactness and weak compactness for a Hankel operator on $BMOA$ are identical.
\end{abstract}
\maketitle

\section{Introduction and notation}

\par We write $z\in\mathbb{D}$, where $\mathbb{D}$ is the unit disc in the complex plane, and $\zeta\in\mathbb{T}=\partial\mathbb{D}$,
where $\mathbb{T}$ is the unit circle. The usual Lebesgue spaces for $\mathbb{T}$ are denoted by $L^p=L^p(\mathbb{T})$ and we write $f(\zeta)\sim
\sum_{n=-\infty}^{+\infty}\widehat{f}(n)\zeta^n$ for the Fourier series of a function $f$ in $L^1$. The Hardy spaces for $\mathbb{T}$ are defined by
$H^p=\{f\in L^p\,:\,\widehat{f}(n)=0\,\text{for}\,n<0\}$. The M. Riesz theorem says that the Riesz projection $P$, defined by
$$Pf(\zeta)\sim\sum_{n=0}^{+\infty}\widehat{f}(n)\zeta^n$$
for $f(\zeta)\sim\sum_{n=-\infty}^{+\infty}\widehat{f}(n)\zeta^n$, is a bounded operator $L^p\to H^p$ when $1<p<\infty$. The Szeg\"o projection or
Cauchy transform of $f$ at $z\in\mathbb{D}$ is defined by
$$Pf(z)=\frac 1{2\pi i}\int_{\mathbb{T}}\frac{f(\zeta)}{\zeta-z}d\zeta.$$
For $1\leq p<+\infty$ and every $f\in L^p$ the $\lim_{r\to 1-}Pf(r\zeta)$ exists for a.e. $\zeta\in\mathbb{T}$ and, when $1<p<+\infty$, this limit is
equal to $Pf(\zeta)$ (where $P$ is the Riesz projection) in both the a.e. sense and in the $L^p$ sense. If $p=1$, the $Pf(\zeta)=\lim_{r\to
1-}Pf(r\zeta)$ serves as the definition of $Pf$, which belongs to the space $L^{1,w}$ of weak-$L^1$ functions. In all cases $Pf(z)$ is an analytic
function of $z\in\mathbb{D}$.
\par A function $f$ is in $BMO$ if $f\in L^1$ and
$$\|f\|_*=\sup_I\frac 1{|I|}\int_I|f(\zeta)-f_I||d\zeta|<+\infty,$$
where $I$ is the general arc of $\mathbb{T}$, $f_I=\frac 1{|I|}\int_If(\zeta)|d\zeta|$ and $|I|$ is the length of $I$. $BMO$ is a Banach space with
the norm $\|f\|_{BMO}=|\widehat{f}(0)|+\|f\|_*$. The space $BMOA=BMO\cap H^1=\{f\in BMO\,:\,\widehat{f}(n)=0\,\text{for}\,n<0\}$ consists of all
analytic functions in $BMO$. It is well known that $L^\infty\subseteq BMO\subseteq L^p$ for $1\leq p<+\infty$.
\par The subspace $VMO$ of $BMO$ contains all $f\in L^1$ for which
$$\lim_{|I|\to 0+}\frac 1{|I|}\int_I|f(\zeta)-f_I||d\zeta|=0.$$
We also define $VMOA=VMO\cap H^1$. $VMOA$ is the closure of analytic polynomials in $BMOA$. The Riesz projection is a bounded operator $L^\infty\to
BMOA$ and also $BMO\to BMOA$.
\par We then have the spaces $BMO_{\log}$ and $VMO_{\log}$ and their variants $BMOA_{\log}$ and $VMOA_{\log}$. An $f\in L^1$ is in $BMO_{\log}$ if
$$\|f\|_{**}=\sup_I\frac{\log\frac{4\pi}{|I|}}{|I|}\int_I|f(\zeta)-f_I||d\zeta|<+\infty.$$
$BMO_{\log}$ is a Banach space with the norm $\|f\|_{BMO_{\log}}=|\widehat{f}(0)|+\|f\|_{**}$. We define $BMOA_{\log}=BMO_{\log}\cap H^1$. Clearly,
$BMO_{\log}\subseteq BMO$.
\par The subspace $VMO_{\log}$ of $BMO_{\log}$ contains all $f\in L^1$ for which
$$\lim_{|I|\to 0+}\frac{\log\frac{4\pi}{|I|}}{|I|}\int_I|f(\zeta)-f_I||d\zeta|=0.$$
We also define $VMOA_{\log}=VMO_{\log}\cap H^1$.
\par For each arc $I$ we define $S(I)=\{z\in\mathbb{D}\,:\,0<1-|z|<\frac{|I|}{2\pi}, \frac z{|z|}\in I\}$, the Carleson square with base $I$. A
positive Borel measure $\mu$ in $\mathbb{D}$ is called a Carleson measure if
$$\sup_I\frac{\mu(S(I))}{|I|}<+\infty.$$
It is known that $\mu$ is a Carleson measure if and only if 
$$\iint_{\mathbb{D}}|f(z)|^2d\mu(z)\leq c(\mu)\int_{\mathbb{T}}|f(\zeta)|^2|d\zeta|,\qquad\qquad f\in H^2$$
for some constant $c(\mu)$ and that, if $c(\mu)$ is the smallest such constant,
$$c(\mu)\asymp \sup_I\frac{\mu(S(I))}{|I|},$$
where $A\asymp B$ means that there are two positive numerical constants $c_1$ and $c_2$ so that $c_1\leq \frac AB\leq c_2$.
\par We know that $f\in H^1$ is in $BMOA$ if and only if the Borel measure $|f'(z)|^2(1-|z|^2)dm(z)$ is a Carleson measure and
$$\|f\|_*^2\asymp\sup_I\frac 1{|I|}\iint_{S(I)}|f'(z)|^2(1-|z|^2)dm(z).$$
Similarly, $f\in H^1$ is in $VMOA$ if and only if
$$\lim_{|I|\to 0+}\frac 1{|I|}\iint_{S(I)}|f'(z)|^2(1-|z|^2)dm(z)=0.$$
Analogously, for functions $f$ in $BMOA_{\log}$ we have
$$\|f\|_{**}^2\asymp\sup_I\frac{\log^2\frac{4\pi}{|I|}}{|I|}\iint_{S(I)}|f'(z)|^2(1-|z|^2)dm(z)$$
and for $f$ in $BMOA_{\log}$ 
$$\lim_{|I|\to 0+}\frac{\log^2\frac{4\pi}{|I|}}{|I|}\iint_{S(I)}|f'(z)|^2(1-|z|^2)dm(z)=0.$$
Note that there exists a positive numerical constant $c$ so that $|f(z)|\leq c\|f\|_{BMO}\log\frac 2{1-|z|^2}$ for all $f\in BMOA$ and all
$z\in\mathbb{D}$. Conversely, there exists a positive numerical constant $c$ so that for all $z\in\mathbb{D}$ there exists an $f\in BMOA$ with
$\|f\|_{BMO}=1$ and $|f(z)|\geq c\log\frac 2{1-|z|^2}$.
\par Finally, with the Fefferman-Stein duality induced by the binary form 
$$\langle f,g \rangle=\lim_{r\to 1-}\frac 1{2\pi}\int_{\mathbb{T}}f(r\zeta)g(\overline{\zeta})|d\zeta|=\lim_{r\to 1-}\frac 1{2\pi
i}\int_{\mathbb{T}}f(r\zeta)g(\overline{\zeta})\overline{\zeta}d\zeta,$$
$BMOA$ is isomorphic to $(H^1)^*$ and $H^1$ is isomorphic to $(VMOA)^*$.
\par Let $a\in H^2$ be an analytic symbol with $\widehat{a}(0)=0$. The Hankel operator with symbol $a$ is defined by
$$H_a(f)=P(aJf),$$
where $J$ is defined by $Jf(\zeta)=\overline{\zeta}f(\overline{\zeta})\sim\sum_{n=-\infty}^{+\infty}\widehat{f}(-n-1)\zeta^n$. Note that $J$ turns an
analytic function $f$ into an antianalytic function $Jf$.
\par $H_a$ is well defined for analytic polynomials $f(\zeta)=\sum_{n=0}^N\widehat{f}(n)\zeta^n$. The set of analytic polynomials is dense in
each $H^p$ ($1\leq p<+\infty$) and there are classical results which specify, for every particular value of $p$, the necessary and sufficient
conditions on the symbol $a$ so that these operators are extended as bounded or even compact operators on $H^p$. The situation is described by the
following theorems.
\begin{theorem}
(Nehari, for $p=2$) Let $1<p<+\infty$. Then $H_a$ is bounded on $H^p$ if and only if $a\in BMOA$.
\end{theorem}
\begin{theorem}
(Hartman, for $p=2$) Let $1<p<+\infty$. Then $H_a$ is compact on $H^p$ if and only if $a\in VMOA$.
\end{theorem}
\begin{theorem}
(Janson-Peetre-Semmes, 1984; Cima-Stegenga, 1987; Tolokonnikov, 1987) $H_a$ is bounded on $H^1$ if and only if $a\in BMOA_{\log}$.
\end{theorem}
\begin{theorem}
(Papadimitrakis-Virtanen, 2008) $H_a$ is bounded on $H^1$ if and only if $a\in BMOA_{\log}$, in which case
$$\|H_a\|_{H^1\to H^1}\asymp \|a\|_{BMO_{\log}}.$$ 
\end{theorem}
\begin{theorem}
(Papadimitrakis-Virtanen, 2008) $H_a$ is compact on $H^1$ if and only if $a\in VMOA_{\log}$.
\end{theorem}
\par Of course, because of the dualities between $H^1$, $BMOA$ and $VMOA$, the above results about $H^1$ hold also for $BMOA$ and $VMOA$.

\section{The main result}

\begin{theorem}
Let $H_a$ be bounded on $BMOA$, i.e. $a\in BMOA_{\log}$. Then the following are equivalent:
\begin{enumerate}
\item $H_a$ is weakly compact on $VMOA$.
\item $H_a$ is compact on $VMOA$.
\item $H_a$ is weakly compact on $BMOA$.
\item $H_a$ is compact on $BMOA$.
\item $H_a(BMOA)\subseteq VMOA$.
\item $a\in VMOA_{\log}$. 
\end{enumerate}
\end{theorem}
\begin{proof}
\par General considerations show the equivalences between (1), (3) and (5), between (2) and (4) and that (2) implies (1). Also, Theorem 1.5 shows
the equivalence between (2) or (4) and (6). Therefore, it remains to prove that (5) implies (6).
\par The symbol $a(\zeta)$ of $H_a$ and the function $b(z)$ are connected by
$$b(z)=\frac 1{2\pi i}\int_T\frac{\overline{\zeta}a(\zeta)}{\zeta-z}d\zeta$$
or, equivalently, $b(\zeta)=\overline{\zeta}a(\zeta)$ for the boundary values of $b$.
\par Also, the variable $f(\zeta)$ of $H_a(f)$ and $g(z)$ are connected by
$$g(z)=\frac 1{2\pi i}\int_T\frac{\overline{f(\overline{\zeta})}}{\zeta-z}d\zeta$$
or, $g(\zeta)=\overline{f(\overline{\zeta})}$ for the boundary values of $g$. 
\par For any arc $I$ of $\mathbb{T}$ let $z_I$ be the midpoint of the inner side of the Carleson square $S(I)$. Let $f\in BMOA$ and $a\in
BMOA_{\log}$ or, equivalently, $g\in BMOA$ and $b\in BMOA_{\log}$. Then
$$\|f\|_*=\|g\|_*,\qquad \|a\|_{**}\asymp\|b\|_{**}.$$
\par It is easy to show that
$$H_a(f)'(z)-\overline{g(z)}b'(z)=\frac 1{2\pi i}\int_{\mathbb{T}}\frac{(b(\zeta)-b(z))\overline{(g(\zeta)-g(z))}}{(\zeta-z)^2}d\zeta,$$
from which we get
\begin{equation}
\begin{split}
|H_a(f)'(z)-\overline{g(z_I)}b'(z)|&\leq\frac
1{2\pi}\int_{\mathbb{T}}\frac{|b(\zeta)-b(z)||g(\zeta)-g(z)|}{|\zeta-z|^2}\,|d\zeta|\notag\\
&\relphantom{\leq}{}+|g(z)-g(z_I)||b'(z)|.\notag
\end{split}
\end{equation}
Applying the Cauchy-Schwarz inequality together with standard estimates for functions in $BMOA$ and in $BMOA_{\log}$, we get
$$\frac 1{2\pi}\int_{\mathbb{T}}\frac{|b(\zeta)-b(z)||g(\zeta)-g(z)|}{|\zeta-z|^2}\,|d\zeta|\leq c\|b\|_{**}\|g\|_*\frac 1{(1-|z|^2)\log\frac
2{1-|z|^2}}.$$
In what follows, $c$ denotes a positive numerical constant, not necessarily the same at each occurrence. This, for every arc $I$, implies
\begin{equation}
\begin{split}
\frac 1{|I|}\iint_{S(I)}&|(H_af)'(z)-\overline{g(z_I)}b'(z)|^2(1-|z|^2)dm(z)\notag\\
&\leq c\|b\|_{**}^2\|g\|_*^2\frac 1{|I|}\iint_{S(I)}\frac 1{(1-|z|^2)\log^2\frac 2{1-|z|^2}}dm(z)\notag\\
&\relphantom{|}{}+\frac c{|I|}\iint_{S(I)}|g(z)-g(z_I)|^2|b'(z)|^2(1-|z|^2)dm(z)\notag\\
&=A+B.\notag
\end{split}
\end{equation}
A direct calculation of the integral of the term $A$ gives
$$A\leq c\,\frac 1{\log\frac{4\pi}{|I|}}\,\|b\|_{**}^2\|g\|_*^2.$$
Observing that $|1-\overline{z_I}z|\asymp |I|$ for all $z\in S(I)$ and considering the Borel measure $d\mu(z)$ which is equal to
$|b'(z)|^2(1-|z|^2)dm(z)$ on $S(I)$ and equal to zero on $\mathbb{D}\setminus S(I)$, we find
\begin{equation}
\begin{split}
B&\leq c|I|\iint_{S(I)}\frac{|g(z)-g(z_I)|^2}{|1-\overline{z_I}z|^2}|b'(z)|^2(1-|z|^2)dm(z)\notag\\
&\leq c|I|\sup_J\frac{\mu(S(J))}{|J|}\int_{\mathbb{T}}\frac{|g(\zeta)-g(z_I)|^2}{|1-\overline{z_I}\zeta|^2}|d\zeta|\notag\\
&\leq c\sup_J\frac{\mu(S(J))}{|J|}\int_{\mathbb{T}}|g(\zeta)-g(z_I)|^2\frac{1-|z_I|^2}{|\zeta-z_I|^2}|d\zeta|\notag\\
&\leq c\|g\|_*^2\sup_J\frac{\mu(S(J))}{|J|}.\notag
\end{split}
\end{equation}
Estimating $\frac{\mu(S(J))}{|J|}=\frac{\mu(S(J)\cap S(I))}{|J|}$, we observe that we need only consider arcs $J$ with $J\cap I\neq\emptyset$. If
$|J|>|I|$, then $\frac{\mu(S(J))}{|J|}\leq\frac{\mu(S(I))}{|I|}$. If $|J|\leq |I|$, then $J\subseteq 3I$, where $3I$ is the arc with the same midpoint
as $I$ and with length three times the length of $I$. Hence, in both cases we get 
\begin{equation}
\begin{split}
\sup_J\frac{\mu(S(J))}{|J|}&\leq \sup_{J\subseteq 3I}\frac 1{|J|}\iint_{S(J)}|b'(z)|^2(1-|z|^2)dm(z)\notag\\
&\leq c\sup_{J\subseteq 3I}\frac 1{\log^2\frac{4\pi}{|J|}}\|b\|_{**}^2\leq\frac c{\log^2\frac{4\pi}{|I|}}\|b\|_{**}^2.\notag
\end{split}
\end{equation}
Therefore,
$$B\leq\frac c{\log^2\frac{4\pi}{|I|}}\,\|b\|_{**}^2\|g\|_*^2$$
and, finally, 
$$\frac 1{|I|}\iint_{S(I)}|H_a(f)'(z)-\overline{g(z_I)}b'(z)|^2(1-|z|^2)dm(z)\leq\frac c{\log\frac{4\pi}{|I|}}\|b\|_{**}^2\|g\|_*^2.$$
\par All the previous estimates of the terms $A$ and $B$ are contained in \cite{PV} and played a substantial role in the proofs of Theorems 1.4 and
1.5. We included them here for the sake of completeness. We also changed them slightly to make them more accessible for the present work.
\par Although this is not our immediate purpose, it is straightforward to show that (6) impies (5). Suppose that $b\in VMOA_{\log}$. Then
\begin{equation}
\begin{split}
\frac 1{|I|}&\iint_{S(I)}|H_a(f)'(z)|^2(1-|z|^2)dm(z)\notag\\
&\leq 2|g(z_I)|^2\frac 1{|I|}\iint_{S(I)}|b'(z)|^2(1-|z|^2)dm(z)+\frac c{\log\frac{4\pi}{|I|}}\|b\|_{**}^2\|g\|_*^2\notag\\
&\leq c\|g\|_*^2\frac {\log^2\frac{4\pi}{|I|}}{|I|}\iint_{S(I)}|b'(z)|^2(1-|z|^2)dm(z)+\frac c{\log\frac{4\pi}{|I|}}\|b\|_{**}^2\|g\|_*^2.\notag
\end{split}
\end{equation}
\par Therefore,
$$\lim_{|I|\to 0+}\frac 1{|I|}\iint_{S(I)}|H_a(f)'(z)|^2(1-|z|^2)dm(z)=0$$
and $H_a(f)\in VMOA$.
\par For the converse, let $b\notin VMOA_{\log}$. Then there is some $\delta>0$ and some sequence $(I_n)$ of arcs such that $|I_n|\to 0$ and
$$\frac {\log^2\frac{4\pi}{|I_n|}}{|I_n|}\iint_{S(I_n)}|b'(z)|^2(1-|z|^2)dm(z)\geq\delta$$
for all $n$. The previous estimates imply that
$$\frac
1{|I_n|}\iint_{S(I_n)}|H_a(f)'(z)|^2(1-|z|^2)dm(z)\geq\frac{\delta}2\frac{|g(z_{I_n})|^2}{\log^2\frac{4\pi}{|I_n|}}-\frac
c{\log\frac{4\pi}{|I_n|}}\|b\|_{**}^2\|g\|_*^2.$$
\par We shall construct some $g\in BMOA$ such that
$$\limsup_{n\to+\infty}\frac{|g(z_{I_n})|}{\log\frac{4\pi}{|I_n|}}>0.$$
This will imply that
$$\limsup_{n\to+\infty}\frac 1{|I_n|}\iint_{S(I_n)}|H_a(f)'(z)|^2dm(z)>0$$
and, hence, that $H_a(f)\notin VMOA$.\vspace{1em}\newline
\textbf{The construction of $g$}.\vspace{1em}
\par Taking a subsequence, we may assume that $I_n$ accumulate to some point of $\mathbb{T}$. For simplicity of the formulas we shall replace
$\mathbb{D}$ by the upper halfplane $\mathbb{H}=\{z=x+iy\,:\,y>0\}$ and $\mathbb{T}$ by the real line $\mathbb{R}$ and assume that the intervals $I_n$
of $\mathbb{R}$ accumulate to $0$.
\par Let $z_n=z_{I_n}=x_n+iy_n$, where $y_n=|I_n|\to 0$ and $x_n(\to 0)$ is the midpoint of $I_n$.
\par We consider a fixed function $\phi : (0,+\infty)\to[0,1]$ with the following properties:\vspace{0.5em}\newline
$(i)$ $\phi(x)=0$ for $x\geq 2$,\newline
$(ii)$ $\phi(x)=1$ for $0<x\leq 1$,\newline
$(iii)$ $\phi$ is smooth in $(0,+\infty)$.\vspace{0.5em}\newline
We then extend $\phi : (-\infty,0)\cup(0,+\infty)\to[-1,1]$ so that\vspace{0.5em}\newline
$(iv)$ $\phi$ is odd.\vspace{0.5em}
\par We also consider the Hilbert transform
$$H\phi(x)=p.v.\frac 1{\pi}\int_{\mathbb{R}}\frac{\phi(t)}{x-t}dt$$
and the analytic function
$$(\phi+iH\phi)(z)=\frac 1{\pi i}\int_{\mathbb{R}}\frac{\phi(t)}{z-t}dt,$$
where $z\in\mathbb{H}$. Now it is easy to prove that
$$|(\phi+iH\phi)(z)|\leq\frac 4{|z|^2},\qquad z\in\mathbb{H}, |z|\geq 3$$
and
$$|(\phi+iH\phi)(z)|\geq \log\frac 1{|z|},\qquad z\in\mathbb{H}, |z|\leq\delta$$
for some appropriate fixed $\delta$ such that $0<\delta\leq\frac 12$. For the first, we assume $|z|\geq 3$ and find
\begin{equation}
\begin{split}
|(\phi+iH\phi)(z)|&=\frac 1{\pi}\Big|\int_{\mathbb{R}}\frac{\phi(t)}{z-t}dt\Big|=\frac 1{\pi}\Big|\int_{\mathbb{R}}\Big(\frac 1{z-t}-\frac
1z\Big)\phi(t)dt\Big|\notag\\
&=\frac 1{\pi}\Big|\int_{-2}^2\frac t{z(z-t)}\,\phi(t)dt\Big|\leq\frac 6{\pi|z|^2}\int_0^2t\phi(t)dt\leq\frac 4{|z|^2}.\notag
\end{split}
\end{equation}
For the second, we assume that $|z|\leq\frac 12$ and find
\begin{equation}
\begin{split}
|(\phi+iH\phi)(z)|&\geq\frac 1{\pi}\Big|\int_0^1\frac 1{z-t}dt-\int_{-1}^0\frac 1{z-t}dt\Big|-\frac 1{\pi}\Big|\int_{1\leq|t|\leq
2}\frac{\phi(t)}{z-t}dt\Big|\notag\\
&\geq 2\log\frac 1{|z|}-c.\notag
\end{split}
\end{equation}
\texttt{First case}: Suppose there is some $c>0$ such that
$$|x_n|\leq c\sqrt{y_n}$$
for all $n$. Of course, then $|z_n|\leq c\sqrt{y_n}=c\sqrt{|I_n|}$ for all $n$.
\par Let $g=\phi+iH\phi$. Then $g\in BMOA$, since $\phi\in L^{\infty}$. For large $n$,
$$|g(z_n)|\geq\log\frac 1{|z_n|}\geq\log\frac 1{c\sqrt{|I_n|}}$$
and thus
$$\lim_{n\to+\infty}\frac{|g(z_n)|}{\log\frac 1{|I_n|}}\geq\frac 12.$$
\texttt{Second case}: Suppose that $\limsup_{n\to+\infty}\frac{|x_n|}{\sqrt{y_n}}=+\infty$ and, taking a subsequence, we may assume that
$$\lim_{n\to+\infty}\frac{|x_n|}{\sqrt{y_n}}=+\infty.$$
\par Now, let
$$\phi_n(x)=\phi\Big(\frac{x-x_n}{\sqrt{y_n}}\Big).$$
Then $\phi_n$ is supported in the interval $[x_n-2\sqrt{y_n},x_n+2\sqrt{y_n}]$ and, taking a further subsequence, we may assume that these intervals
are pairwise disjoint. This implies that the function
$$\psi=\sum_{k=1}^{+\infty}\phi_k$$
is in $L^{\infty}$ and, more precisely, $|\psi|\leq 1$ a.e. in $\mathbb{R}$. Finally, we define
$$g(z)=(\psi+iH\psi)(z)=\sum_{k=1}^{+\infty}(\phi+iH\phi)\Big(\frac{z-x_k}{\sqrt{y_k}}\Big),\qquad z\in\mathbb{H}.$$
Since $\psi\in L^{\infty}$, we have that $g\in BMOA$. Using one of the previous estimates, we see that for large $n$
$$\Big|(\phi+iH\phi)\Big(\frac{z_n-x_n}{\sqrt{y_n}}\Big)\Big|=|(\phi+iH\phi)(i\sqrt{y_n})|\geq\log\frac 1{\sqrt{y_n}}=\frac 12\log\frac 1{|I_n|}.$$
We may now assume that $|x_{n+1}|\leq\frac 12|x_n|$ and that $y_{n+1}\leq y_n$ for all $n$. Then for $k<n$ we have
$|\frac{x_n-x_k+iy_n}{\sqrt{y_k}}|\geq|\frac{x_n-x_k}{\sqrt{y_k}}|\geq\frac{|x_k|}{2\sqrt{y_k}}\geq 3$ and for $k>n$ we have
$|\frac{x_n-x_k+iy_n}{\sqrt{y_k}}|\geq|\frac{x_n-x_k}{\sqrt{y_k}}|\geq\frac{|x_n|}{2\sqrt{y_n}}\geq 3$. By our estimates,
$$\sum_{k=1}^{n-1}\Big|(\phi+iH\phi)\Big(\frac{z_n-x_k}{\sqrt{y_k}}\Big)\Big|\leq 4\sum_{k=1}^{n-1}\frac{y_k}{(x_n-x_k)^2+y_n^2}\leq
16\sum_{k=1}^{n-1}\frac{y_k}{x_k^2}$$
and
$$\sum_{k=n+1}^{+\infty}\Big|(\phi+iH\phi)\Big(\frac{z_n-x_k}{\sqrt{y_k}}\Big)\Big|\leq 4\sum_{k=n+1}^{+\infty}\frac{y_k}{(x_n-x_k)^2+y_n^2}\leq
4\sum_{k=n+1}^{+\infty}\frac{y_k}{x_k^2}.$$
Since we may choose the intervals so that $\frac{|x_n|}{\sqrt{y_n}}\to+\infty$ fast enough, we may suppose that
$$m=\sum_{k=1}^{+\infty}\frac{y_k}{x_k^2}<+\infty.$$
Therefore
$$|g(z_n)|\geq\frac 12\log\frac 1{|I_n|}-16m$$
and, finally,
$$\lim_{n\to+\infty}\frac{|g(z_n)|}{\log\frac 1{|I_n|}}\geq\frac 12.$$
\end{proof}


\begin{thebibliography}{99}

\bibitem{CS}
J. Cima \and D. Stegenga, {\it Hankel operators on $H^p$}. Analysis at Urbana. Vol. 1: Analysis in function spaces, London Mathematical Society
Lecture Note Series, 137, Cambridge University
Press, Cambridge (1989), 133--150.

\bibitem{Jan}
S. Janson, {\it On functions with conditions on the mean
oscillation}. Ark. Mat. \textbf{14}, N2 (1976), 189--196.

\bibitem{JPS}
S. Janson, J. Peetre \and S. Semmes, {\it On the action of Hankel and
Toeplitz operators on some function spaces}. Duke Math. J. \textbf{51},
no. 4 (1984), 937--958.

\bibitem{Steg}
D. A. Stegenga, {\it Bounded Toeplitz operators on $H^1$ and
applications of the duality between $H^1$ and the functions of
bounded mean oscillation}. Amer. J. of Math. \textbf{98}, no. 3 (1976),
573--589.

\bibitem{Tol}
V. A. Tolokonnikov, {\it Hankel and Toeplitz operators in Hardy
spaces} (Russian. English summary). Investigations on linear
operators and the theory of functions, XIV, {\em Zap. Nauchn. Sem.
Leningrad. Otdel. Mat. Inst. Steklov. (LOMI)} 141 (1985), 165--175.
(English translation in J. Sov. Math. \textbf{37} (1987),
1359--1364.)

\bibitem{PV}
M. Papadimitrakis \and J. Virtanen, {\it Hankel and Toeplitz transforms on $H^1$: continuity, compactness and Fredholm properties}. Integral Equations
and Operator Theory \textbf{61}, (2008), 573--591.

\end{thebibliography}
\end{document}